\documentclass{amsart}

\usepackage{amssymb,latexsym,amsfonts,amsmath}
\usepackage{graphicx,psfrag,epsfig,epsf}
\include{diagrams}
\usepackage{eso-pic}
\usepackage{graphicx}
\usepackage{color}
\usepackage{diagrams}
\usepackage{algorithm2e}
\usepackage{tikz}
\usetikzlibrary{automata}

\usepackage{amsfonts}
\usepackage{amssymb,latexsym,amsfonts,amsmath}
\usepackage{graphicx}

\newtheorem{theorem}{Theorem}[section]
\newtheorem{lemma}[theorem]{Lemma}

\newtheorem{proposition}[theorem]{Proposition}

\theoremstyle{definition}
\newtheorem{definition}[theorem]{Definition}

\theoremstyle{remark}

\numberwithin{equation}{section}

\begin{document}

\begin{abstract}                          
In this paper we show that incrementally stable nonlinear time--delay systems admit symbolic models which are approximately equivalent, in the sense of approximate bisimulation, to the original system. An algorithm is presented which computes the proposed symbolic models. Termination of the algorithm in a finite number of steps is guaranteed by a boundedness assumption on the state and input spaces of the system.
\end{abstract}

\title[Symbolic models for nonlinear time-delay systems using approximate bisimulations]{Symbolic models for nonlinear time-delay systems using approximate bisimulations}
\thanks{This work has been partially supported by the Center of Excellence for Research DEWS, University of L'Aquila, Italy and by the National Science Foundation CAREER award 0717188.}

\author[Giordano Pola, Pierdomenico Pepe, Maria D. Di Benedetto and Paulo Tabuada]{
Giordano Pola$^{1}$, Pierdomenico Pepe$^{1}$, Maria D. Di Benedetto$^{1}$ and Paulo Tabuada$^{2}$}
\address{$^{1}$
Department of Electrical and Information Engineering, Center of Excellence DEWS,
University of L{'}Aquila, Poggio di Roio, 67040 L{'}Aquila, Italy}
\email{ \{giordano.pola,pierdomenico.pepe,mariadomenica.dibenedetto\}@univaq.it}
\urladdr{
http://www.diel.univaq.it/people/pola/
}
\urladdr{
http://www.diel.univaq.it/people/pepe/
}
\urladdr{
http://www.diel.univaq.it/people/dibenedetto/
}

\address{$^{2}$Department of Electrical Engineering\\
University of California at Los Angeles,
Los Angeles, CA 90095}
\email{tabuada@ee.ucla.edu}
\urladdr{http://www.ee.ucla.edu/$\sim$tabuada}

\maketitle

\section{Introduction}

Symbolic models have been the object of intensive study in the last few years since they provide a tool for mitigating complexity in the analysis and control of large scale systems \cite{TACsymbolicmodels}. In particular, they enable a correct-–by-–design approach to the synthesis of embedded control software, see e.g. \cite{LTLControl,TabuadaTAC08}. The key idea in this approach is to regard the synthesis of software as a control problem to be solved in conjunction with the synthesis of the control algorithms. 
Central to this approach is the possibility to construct symbolic models that approximately describe continuous control systems. Symbolic models are abstract mathematical models where each symbolic state and each symbolic label represent an aggregation of continuous states and an aggregation of input signals in the original continuous model. 
Many researchers have recently faced the problem of identifying classes of control systems admitting symbolic models. For example, controllable linear control systems and incrementally stable nonlinear control systems were shown in \cite{LTLControl} and respectively in \cite{PolaAutom2008}, to admit symbolic models. In the work of \cite{Belta:06} symbolic models for multi--affine systems are proposed and benefits from their use in solving control problems arising in systems biology and robot motion planning have been shown in \cite{Batt:08} and in \cite{Belta:05}, respectively. Nonlinear switched systems have been considered in \cite{Girard_et_al_HSCC08} and applications to the digital control of the boost DC--DC converter have been investigated. One challenge in this research line is to enlarge the class of systems admitting symbolic models. In this paper we make a further step along this direction by focusing on the class of time--delay systems. Time--delay systems are an important class of dynamical systems, arising in many application domains of interest ranging from biology, chemical, electrical, and mechanical engineering, to economics (see e.g. \cite{Niculescu:01,surveyTDS,TDSAQ}). \\
In this paper we generalize the results of the work in \cite{PolaAutom2008} from nonlinear control systems to nonlinear time--delay systems. The main contribution of this paper lies in showing that incrementally stable nonlinear time--delay systems do admit symbolic models. An algorithm is presented which computes the proposed symbolic models in a finite number of steps, provided that the sets of states and inputs of the time--delay system are bounded. 
In addition to theoretical relevance, the importance of this result resides in the offer of an alternative design methodology 
based on symbolic models\footnote{See e.g. \cite{LTLControl,TabuadaTAC08} for symbolic models--based control of linear and nonlinear control systems.} to the control design of \textit{nonlinear} time--delay systems, which is at the present a difficult task to deal with, by using current methodologies \cite{surveyTDS}. \\
In the following we will use a notation which is standard within both the control and computer science community. However for the sake of completeness, a detailed list of the employed notation is included in the Appendix (Section \ref{sec:notation}).

\section{Time--Delay Systems} 

In this paper we consider the following nonlinear time--delay system:
\begin{equation}
\label{TDS}
\left\{
\begin{array}{lll}
\dot{x}(t)=f(x_{t},u(t-r)),& t \in \mathbb{R}^{+}, a.e. \\
x(t)=\xi_{0}(t), & t \in [-\Delta,0], 
\end{array}
\right.
\end{equation}
where $\Delta\in \mathbb{R}^{+}_{0}$ is the maximum involved state delay, $r\in\mathbb{R}^{+}_{0}$ is the input delay, $x(t)\in X \subseteq \mathbb{R}^{n}$, $x_{t}\in \mathcal{X}\subseteq C^{0}([-\Delta,0];X)$, $u(t)\in U\subseteq \mathbb{R}^{m}$ is the control input at time $t\in [-r,+\infty [$ , $\xi_{0}\in \mathcal{X}$ is the initial condition, $f$ is a functional from $\mathcal{X}\times U$ to $\mathcal{X}$. We denote by $\mathcal{U}$ the class of control input signals and we suppose that $\mathcal{U}$ is a subset of the set of all measurable and locally essentially bounded functions of time from $[-r,+\infty[$ to $U$. Moreover we suppose that $f$ is Lipschitz on bounded sets, i.e. for every bounded set
$K\subset\mathcal{X} \times U$, there exists a constant $\kappa>0$ such that 
\[
\Vert f(x_{1},u_{1})-f(x_{2},u_{2})\Vert\leq \kappa(\Vert x_{1}-x_{2}\Vert_{\infty} + \Vert u_{1}-u_{2} \Vert),
\]
for all $(x_{1},u_{1}),(x_{2},u_{2})\in K$. Without loss of generality we assume \mbox{$f(0,0) = 0$}, thus ensuring that $x(t) = 0$ is the trivial solution for the unforced system $\dot{x}(t) = f(x_{t}, 0)$. Multiple discrete non--commensurate as well as distributed delays
can appear in (\ref{TDS}). Assumptions on $f$ ensure existence and uniqueness of the solutions of the differential equation in (\ref{TDS}). 
In the following $x(t,\xi_0,u)$ and $x_t(\xi_0,u)$ will denote the solutions in $X$ and respectively in $\mathcal{X}$, of the time--delay system with initial condition $\xi_0$ and input $u\in \mathcal{U}$, at time $t$. A time--delay system is said to be \textit{forward complete} if every solution is defined on $[0,+\infty [$. In the further developments we refer to a time--delay system as in (\ref{TDS}) by means of the tuple:
\[
\Sigma=(X,\mathcal{X},\xi_{0},U,\mathcal{U},f),
\]
where each entity has been defined before.

\section{Incremental Stability}\label{sec:stab}

The results presented in this paper will assume certain stability assumptions that we introduce in this section. The following definition has been obtained as a natural generalization of the one in \cite{IncrementalS}.

\begin{definition}

A time--delay system $\Sigma=(X,\mathcal{X},\xi_{0},U,\mathcal{U},f)$ is 
\textit{incrementally input--to--state stable} (\mbox{$\delta$--ISS}) if it is
forward complete and there exist a $\mathcal{KL}$ function $\beta$ and a
$\mathcal{K}$ function $\gamma$ such that for any time $t\in\mathbb{R}^{+}_{0}$, any initial conditions
$\xi_{1},\xi_{2}\in\mathcal{X}$ and any inputs $u_{1},u_{2}\in\mathcal{U}$ the following inequality holds:
\begin{equation}
\left\Vert x_{t}(\xi_{1},u_{1})-x_{t}(\xi_{2}%
,u_{2})\right\Vert_{\infty} \leq \beta(\left\Vert \xi_{1}-\xi_{2}\right\Vert_{\infty}
,t)+\gamma(\left\Vert (u_{1}-u_{2})|_{[-r,t-r)} \right\Vert _{\infty}).
\label{deltaISS}
\end{equation}
\end{definition}
The above definition can be thought of as an incremental version
of the notion of input--to--state stability (ISS). 
In general, inequality in (\ref{deltaISS}) is difficult
to check directly. We therefore provide hereafter a characterization of $\delta$--ISS, in terms of Liapunov--Krasovskii functionals (see \cite{PepeJiangSCL}, as far as the ISS is concerned).

\begin{definition}
Given a time--delay system $\Sigma=(X,\mathcal{X},\xi_{0},U,\mathcal{U},f)$, a locally Lipschitz functional 
$
V:C^{0}([-\Delta,0];\mathbb{R}^{n})\times C^{0}([-\Delta,0];\mathbb{R}^{n})\to \mathbb{R}^{+}
$ 
is said to be a $\delta$--ISS Liapunov--Krasovskii functional for $\Sigma$ if there exist $\mathcal{
K}_{\infty}$ functions $\alpha_1, \alpha_2$ and $\mathcal{K}$ functions $\alpha_3$, $\rho$ such that:

\begin{itemize}

\item[(i)] for all $x_{1},x_{2}\in C^{0}([-\Delta,0];\mathbb{R}^{n})$
\[
\alpha_1(\Vert x_{1}(0)-x_{2}(0) \Vert)\le V(x_1,x_2)\le \alpha_2(M_a(x_1-x_2)),
\]
where $M_{a}:C^{0}([-\Delta,0];\mathbb{R}^{n}) \to \mathbb{R}^+$ is a continuous functional such that 
\[
\underline\gamma_{a}(\Vert x(0)\Vert )\le M_{a}(x)\le \overline \gamma_{a}(\Vert x\Vert_{\infty}), \ \forall x\in C^{0}([-\Delta,0];\mathbb{R}^{n}),\nonumber
\]
for some $\mathcal{K}_{\infty}$ functions $\underline \gamma_{a}$ and
$\overline \gamma_{a}$;

\item[(ii)] for all $x_{1},x_{2}\in C^{0}([-\Delta,0];\mathbb{R}^{n})$ and $u_{1},u_{2}\in \mathbb{R}^{m}$ for which \mbox{$M_a(x_1-x_2)\ge$} 
$\rho(\Vert u_1-u_{2}\Vert)$ the following inequality holds:
\[
D^+V(x_1,x_2,u_1,u_2)\le -\alpha_3(M_a(x_1-x_2)),
\]
where $D^+V(x_1,x_2,u_1,u_2)$ is the derivative of functional $V$ in the formulation proposed by Driver \cite{Driver62}, i.e.
\[
D^+V(x_1,x_{2},u_{1},u_{2})=\limsup_{\theta\to 0^+} \frac {V(x^\theta_1,x^\theta_2)-V(x_1,x_2)}{\theta}, 
\]
where $x_i^\theta(s)=x_i(s+\theta)$, if $s\in [-\Delta,-\theta[$ and $x_i^\theta(s)=x_i(0)+$\mbox{$(s+\theta)f(x_i,u_i)$}, if $s\in [-\theta,0]$.
%
\end{itemize}
\end{definition}


\begin{theorem} 
A time--delay system $\Sigma$ is $\delta$--ISS if it admits a $\delta$--ISS Liapunov--Krasovskii functional.
\end{theorem}
\begin{proof}
As pointed out in \cite{IncrementalS}, the same lines of the proof
used by Sontag for ISS, used also for time--delay systems in
\cite{PepeJiangSCL}, can be used here. 
Briefly, by results in \cite{PepeTACABS,PepeSAFEAUTOMATICA}
, let
$\phi_1,\phi_2 \in C^{1}([-\Delta,0];\mathbb{R}^{n})$ be a pair of initial conditions,
$u_1,u_2$ a pair of input functions. Let $\Vert
u_1-u_2\Vert_{\infty} = v$. It can be proved that the set 
$
S=\{(\psi_1,\psi_2)\in $ $C^{0}([-\Delta,0];\mathbb{R}^{n})\times C^{0}([-\Delta,0];\mathbb{R}^{n}):
V(\psi_1,\psi_2)\le \alpha_2\circ \rho(v) \}
$ 
is forward invariant, i.e., if $(x_{t_0}(\phi_1,u_1),
x_{t_0}(\phi_2,u_2))\in S$ for some $t_0\in\mathbb{R}^{+}_{0}$, then \mbox{$(x_{t}(\phi_1,u_1),$} $x_{t}(\phi_2,u_2))\in S$ for all $t\ge t_0$. In the
interval $[0,t_0)$ with $t_{0}\in\mathbb{R}^{+}$, where
$(x_{t}(\phi_1,u_1), x_{t}(\phi_2,u_2))$, eventually, does not
belong to $S$, the inequality in (ii) holds, which results for
$w(t)=V(x_t(\phi_1,u_1),$ $x_t(\phi_2,u_2))$ in $D^+w(t)\le -\alpha_3\circ \alpha_2^{-1}(w(t))$ a.e., 
from which, by inequalities in (i), the following inequality
holds, for a suitable $\mathcal{KL}$ function $\bar \beta$,
\begin{eqnarray}
&&
\Vert x(t,\phi_1,u_1)-x(t,\phi_2,u_2)\Vert\le \nonumber \alpha_1^{-1}\circ \bar \beta (\alpha_2\circ \overline
\gamma_a(\Vert \phi_1-\phi_2 \Vert_{\infty}),t).
\end{eqnarray}
By the result concerning the set $S$, the following inequality holds:
\[
\Vert x(t,\phi_1,u_1)-x(t,\phi_2,u_2)\Vert \le \alpha_1^{-1}\circ \bar \beta (\alpha_2\circ \overline
\gamma_a(\Vert \phi_1-\phi_2\Vert_{\infty}),t)+\alpha_1^{-1}\circ \alpha_2\circ \rho(v).
\]
From the above inequality, one gets:
\[
\begin{array}
{rcl}
\Vert x_t(\phi_1,u_1)-x_t(\phi_2,u_2)\Vert_{\infty}  & \le & e^{-(t-\Delta)}\Vert \phi_1-\phi_2\Vert_{\infty}\nonumber\\
&
+
&
\alpha_1^{-1}\circ \bar \beta (\alpha_2\circ \overline
\gamma_a(\Vert \phi_1-\phi_2\Vert_{\infty}),\max
\{0,t-\Delta\})\nonumber\\
&
+
&
\alpha_1^{-1}\circ \alpha_2\circ \rho(v)
\nonumber
\end{array}\nonumber
\]
and by causality arguments, the inequality in (\ref{deltaISS}) is proved.
\end{proof}

At the present it is not known whether existence of $\delta$--ISS Liapunov--Krasovskii functional is also a necessary condition for a time--delay system to be $\delta$--ISS. Sufficient and necessary conditions for a time--delay system to be ISS, in terms of existence of ISS Liapunov--Krasovskii functionals can be found in \cite{PepeEJC08}.

\section{Symbolic Models and Approximate Equivalence\label{sec4}}
In this paper we use transition systems as abstract mathematical models of time--delay systems. 
\begin{definition}
A transition system is a sixtuple 
$
\mbox{$T=(Q,q_{0},L,\rTo,O,H)$}, 
$ 
consisting of:
\begin{itemize}
\item A set of states $Q$;
\item An initial state $q_{0}\in Q$;
\item A set of labels $L$;
\item A transition relation $\rTo\subseteq Q\times L\times Q$;
\item An output set $O$;
\item An output function $H:Q\rightarrow O$.
\end{itemize}
A transition system $T$ is said to be: 
\textit{metric}, if the output set $O$ is equipped with a metric \mbox{$\mathbf{d}:O\times O\rightarrow\mathbb{R}_{0}^{+}$}; 
\textit{countable}, if $Q$ and $L$ are countable sets; 
\textit{finite/symbolic}, if $Q$ and $L$ are finite sets. 
\end{definition}
We will follow standard practice and denote an element $(q,l,p)\in
\rTo$ by $q\rTo^{l} p$. Transition systems
capture dynamics through the transition relation. For any states \mbox{$q,p\in Q$},
$q\rTo^{l} p$ simply means that it is possible to evolve from state $q$ to state $p$ under the action labeled by $l$. 
In this paper we will show how to construct symbolic models that are approximately equivalent to $\Sigma$. The notion of 
equivalence that we consider is the one of \textit{bisimulation equivalence} \cite{Milner,Park}. Bisimulation relations are standard mechanisms to relate
the properties of transition systems. Intuitively, a
bisimulation relation between a pair of transition systems $T_{1}$ and $T_{2}$
is a relation between the corresponding sets of states explaining how a state trajectory $s_{1}$ of $T_{1}$ can be transformed into a state 
trajectory $s_{2}$ of $T_{2}$ and vice versa. While typical bisimulation
relations require that $s_{1}$ and $s_{2}$ are observationally
indistinguishable, that is $H_{1}(s_{1})=H_{2}(s_{2})$, we shall relax this
by requiring $H_{1}(s_{1})$ to simply be close to $H_{2}(s_{2})$ where
closeness is measured with respect to the metric on the output set. The
following notion has been introduced in \cite{AB-TAC07} and in a
slightly different formulation in \cite{TabuadaTAC08}.

\begin{definition}
\label{ASR}Let $T_{1}=(Q_{1},q_{1}^{0},L_{1},\rTo_{1},O,H_{1})$ and
$T_{2}=(Q_{2},q_{2}^{0},L_{2},\rTo_{2},$ $O,H_{2})$ be metric transition systems
with the same output set $O$ and metric $\mathbf{d}$, and let \mbox{$\varepsilon \in \mathbb{R}_{0}^{+}$} 
be a given precision.
A relation $R\subseteq Q_{1}\times Q_{2}$ is said to be an \mbox{$\varepsilon
$--\textit{approximate}} bisimulation relation between $T_{1}$ and $T_{2}$, if for
any $(q_{1},q_{2})\in R$:
\begin{itemize}
\item[(i)] $\mathbf{d}(H_{1}(q_{1}),H_{2}(q_{2}))\leq\varepsilon$;
\item[(ii)] $q_{1}\rTo^{l_{1}}_{1} p_{1}$ implies 
existence of $q_{2}\rTo^{l_{2}}_{2} p_{2}$ such that
$(p_{1},p_{2})\in R$;
\item[(iii)] $q_{2}\rTo^{l_{2}}_{2} p_{2}$ implies 
existence of $q_{1}\rTo^{l_{1}}_{1} p_{1}$ such that
$(p_{1},p_{2})\in R$.
\end{itemize}
Moreover $T_{1}$ is said to be $\varepsilon$\textit{--bisimilar} to $T_{2}$ if:
\begin{itemize}
\item [(iv)]there exists an \mbox{$\varepsilon$--approximate} bisimulation relation $R$\ between $T_{1}$ and $T_{2}$ such that \mbox{$(q_{1}^{0},q_{2}^{0})\in R$}.
\end{itemize}
\end{definition}

\section{Approximately Bisimilar Symbolic Models\label{sec6}}
In this paper we consider time--delay systems with digital controllers, i.e. time--delay systems where control inputs are piecewise--constant. In many concrete applications controllers are implemented through digital devices and this motivates our interest for this class of control systems. 
In the following we refer to time--delay systems with digital controllers as \textit{digital time--delay systems}. 
From now on we suppose that the set $U$ of input values 
contains the origin and that it is a hyper rectangle of the form $U:= [a_1,b_1] \times [a_2,b_2] \times ... \times [a_m,b_m]$, for some $a_i<b_i, i=1,2,...,m$. Furthermore given \mbox{$\tau\in\mathbb{R}^{+}$}, 
we consider the following class of control inputs:%
\begin{equation}
\mathcal{U}_{\tau}:=
\left\{
\begin{array}
[c]{c}
u\in\mathcal{U}: \textrm{the time domain of } u \textrm{ is } [-r,-r+\tau] \\
\textrm{ and } u(t)=u(-r),t\in\lbrack -r,-r+\tau]
\end{array}
\right\}.
\label{Utau}
\end{equation}
Given $k\in\mathbb{R}^{n}$ we denote by $\mathcal{U}_{k,\tau}$ the class of control inputs obtained by the concatenation of $k$ control inputs in $\mathcal{U}_{\tau}$. Given a digital time--delay system $\Sigma$ define the transition system 
$
T_{\tau}(\Sigma):=(Q_{1},q_{1}^{0},L_{1},\rTo_{1},O_{1},H_{1}),
$ 
where:
\begin{itemize}
\item $Q_{1}=\mathcal{X}$;
\item $q_{1}^{0}=\xi_{0}$;
\item $L_{1}=\{l_{1}\in \mathcal{U}_{\tau}\,\,\vert\,\,x_\tau (x,l_{1})$ is defined for all $x\in\mathcal{X}\}$;
\item $q \rTo_{1}^{l_{1}} p$, if 
$x_{\tau}(q,l_{1})=p$;
\item $O_{1}=\mathcal{X}$;
\item $H_{1}=1_{\mathcal{X}}$.
\end{itemize}
Transition system $T_{\tau}(\Sigma)$ can be thought of as a time discretization of $\Sigma$. Transition system $T_{\tau}(\Sigma)$ is metric when we regard
$O_{1}=\mathcal{X}$ as being equipped with the metric \mbox{$\mathbf{d}(p,q)=\Vert p-q\Vert_{\infty}$}.
Note that transition system $T_{\tau}(\Sigma)$ is not symbolic, since the set of states $Q_{1}$ is a functional space. 
The construction of symbolic models for digital time--delay systems relies upon approximations of the set of reachable states and of the space of input signals. Given a digital time--delay system $\Sigma$ let $R_{\tau}(\Sigma)\subseteq \mathcal{X}$ be the set of reachable states of $\Sigma$ at times $t=0,\tau,...,k\tau,...$, i.e. the collection of all states $x\in\mathcal{X}$ for which there exist $k\in \mathbb{N}$ and a control input $u\in\mathcal{U}_{k,\tau}$ so that $x=x_{k\tau}(\xi_{0},u)$. The sets $R_{\tau}(\Sigma)$ and $\mathcal{U}_{\tau}$, corresponding to\footnote{In fact the set $Q_{1}$ of states of $T_{\tau}(\Sigma)$ is $\mathcal{X}$ and not $R_{\tau}(\Sigma)$. However, all states in $\mathcal{X} \backslash R_{\tau}(\Sigma)$ will be never reached and this is the reason why we will approximate $R_{\tau}(\Sigma)$ rather than $\mathcal{X}$.} $Q_{1}$ and $L_{1}$ in $T_{\tau}(\Sigma)$ are functional spaces and therefore are needed to be approximated, in the sense of the following definition.
\begin{definition}
\label{CountApprox}
Consider a functional space $\mathcal{Y}\subseteq C^{0}(I,Y)$ with $Y\subseteq\mathbb{R}^{n}$, \mbox{$I=[a,b]$}, $a,b\in \mathbb{R}$, $a<b$. A map \mbox{$\mathcal{A}:\mathbb{R}^{+}\rightarrow 2^{C^{0}(I,Y)}$} is a \textit{countable approximation} of $\mathcal{Y}$ if for any desired precision $\lambda\in\mathbb{R}^{+}$:
\begin{itemize}
\item [(i)] $\mathcal{A}(\lambda)$ is a countable set;
\item [(ii)] for any $y\in \mathcal{Y}$ there exists $z\in \mathcal{A}(\lambda)$ so that $\Vert y - z \Vert_{\infty} \leq \lambda$;
\item [(iii)] for any $z\in \mathcal{A}(\lambda)$ there exists $y\in \mathcal{Y}$ so that $\Vert y - z \Vert_{\infty} \leq \lambda$.
\end{itemize}
\end{definition}
A countable approximation $\mathcal{A}_{\mathcal{U}}$ of $\mathcal{U}_{\tau}$ can be easily obtained by defining for any $\lambda_{\mathcal{U}}\in\mathbb{R}^{+}$,
\begin{equation}
\mathcal{A}_{\mathcal{U}}(\lambda_{\mathcal{U}})=\{u\in\mathcal{U}_{\tau}: u(t)=u(-r)\in [U]_{2\lambda_{\mathcal{U}}} ,t\in\lbrack -r,-r+\tau]\},
\label{UtauQ}
\end{equation}
where $[U]_{2\lambda_{\mathcal{U}}}$ is defined as in (\ref{grid}). By comparing $\mathcal{U}_{\tau}$ in (\ref{Utau}) and $\mathcal{A}_{\mathcal{U}}(\lambda_{\mathcal{U}})$ in (\ref{UtauQ}) it is readily seen that $\mathcal{A}_{\mathcal{U}}(\lambda_{\mathcal{U}})\subset \mathcal{U}_{\tau}$ for any $\lambda_{\mathcal{U}}\in\mathbb{R}^{+}$. Under assumptions on $U$, the set $\mathcal{A}_{\mathcal{U}}(\lambda_{\mathcal{U}})$ is nonempty\footnote{For any $\lambda_{\mathcal{U}}\in\mathbb{R}^{+}$ the set $\mathcal{A}_{\mathcal{U}}(\lambda_{\mathcal{U}})$ contains at least the identically null input function.} for any $\lambda_{\mathcal{U}}\in\mathbb{R}^{+}$. The definition of countable approximations of the set of reachable states $R_{\tau}(\Sigma)$ is more involved since $R_{\tau}(\Sigma)$ is a functional space. Let us assume as a first step existence of a countable approximation $\mathcal{A}_{\mathcal{X}}$ of $R_{\tau}(\Sigma)$. (In the further development we will derive conditions ensuring existence and construction of $\mathcal{A}_{\mathcal{X}}$.) 
We now have all the ingredients to define a countable transition system that will approximate $T_{\tau}(\Sigma)$. 
Given any $\tau\in\mathbb{R}^{+}$, $\lambda_{\mathcal{X}}\in\mathbb{R}^{+}$ 
and $\lambda_{\mathcal{U}}\in\mathbb{R}^{+}$ define the following
transition system:
\begin{equation}
T_{\tau,\lambda_{\mathcal{X}},\lambda_{\mathcal{U}}}(\Sigma):=(Q_{2},q_{2}^{0},L_{2},\rTo_{2},O_{2},H_{2}),
\label{T2}%
\end{equation}
where:
\begin{itemize}
\item $Q_{2}=\mathcal{A}_{\mathcal{X}}(\lambda_{\mathcal{X}})$;
\item $q_{2}^{0}\in Q_{2}$ so that $\Vert \xi_{0}-q_{2}^{0} \Vert_{\infty}\leq\lambda_{\mathcal{X}}$;
\item $L_{2}=\mathcal{A}_{\mathcal{U}}(\lambda_{\mathcal{U}})$; 
\item $q \rTo^{l}_{2} p$, if $\left\Vert p-x_{\tau}(q,l)\right\Vert_{\infty} 
\leq \lambda_{\mathcal{X}}$;
\item $O_{2}=\mathcal{X}$;
\item $H_{2}=\imath : Q_{2} \hookrightarrow O_{2}$.
\end{itemize}
Parameters $\lambda_{\mathcal{X}}$ and $\lambda_{\mathcal{U}}$ can be thought of as quantizations of the set $R_{\tau}(\Sigma)$ and of the space $\mathcal{U}_{\tau}$, respectively. 
By construction, transition system in (\ref{T2}) is countable. We can now state the following result that relates $\delta$--ISS to the
existence of symbolic models for time--delay systems.
\begin{theorem}
Consider a digital time--delay system $\Sigma=(X,\mathcal{X},\xi_{0},U,\mathcal{U}_{\tau},f)$ and any desired precision $\varepsilon\in\mathbb{R}^{+}$. Suppose that 
$\Sigma$ is $\delta$--ISS and choose $\tau\in\mathbb{R}^{+}$ so that $\beta(\varepsilon,\tau)<\varepsilon$. Moreover suppose that there exists a countable approximation $\mathcal{A}_{\mathcal{X}}$ of $R_{\tau}(\Sigma)$. Then, for any $\lambda_{\mathcal{X}}\in\mathbb{R}^{+}$ and $\lambda_{\mathcal{U}}\in\mathbb{R}^{+}$ satisfying the following
inequality:%
\begin{eqnarray}
\beta(\varepsilon,\tau)+\gamma(\lambda_{\mathcal{U}})+ \lambda_{\mathcal{X}} \leq \varepsilon
\label{cond2}
\end{eqnarray}
transition systems $T_{\tau,\lambda_{\mathcal{X}},\lambda_{\mathcal{U}}}(\Sigma)$ and $T_{\tau}(\Sigma)$ are $\varepsilon
$--bisimilar.
\label{ThMain}
\end{theorem}
\begin{proof}
The proof can be given along the lines of Theorem 5.1 in \cite{PolaAutom2008}. We include it here for the sake of completeness. Consider the relation \mbox{$R\subseteq Q_{1}\times Q_{2}$} defined by $(x,q)\in R$ if and only if \mbox{$\Vert H_{1}(x)-H_{2}(q)\Vert_{\infty}\leq \varepsilon$}. 
We now show that $R$ is an \mbox{$\varepsilon$--approximate} bisimulation relation between
$T_{\tau}(\Sigma)$ and $T_{\tau,\lambda_{\mathcal{X}},\lambda_{\mathcal{U}}}(\Sigma)$.
Consider any $(x,q)\in R$. Condition (i) in Definition \ref{ASR} is satisfied
by the definition of $R$. Let us now show that condition (ii) in Definition
\ref{ASR} holds. Consider any $l_{1}\in L_{1}$ and the transition 
$x \rTo^{l_{1}}_{1} y$ in $T_{\tau}(\Sigma)$. By definition of $L_{2}$ there exists $l_{2}\in L_{2}$ so that:
\begin{equation}
\left\Vert l_{1} - l_{2}\right\Vert_{\infty} \leq \lambda_{\mathcal{U}}.
\label{no3}
\end{equation}
Set $z=x_{\tau}(q,l_{2})$. Note that since $l_{2}\in L_{2}\subseteq \mathcal{U}_{\tau}$, function $z$ is well defined and $z\in R_{\tau}(\Sigma)$. By definition of $Q_{2}$ there exists $p\in Q_{2}$ so that:
\begin{equation}
\left\Vert z-p \right\Vert_{\infty} \leq \lambda_{\mathcal{X}}.
\label{no4}
\end{equation}
By the above inequality it is clear that $q \rTo^{l_{2}}_{2} p$ in $T_{\tau,\lambda_{\mathcal{X}},\lambda_{\mathcal{U}}}(\Sigma)$. Since $\Sigma$ is $\delta$--ISS
and by (\ref{cond2}), (\ref{no3}) and (\ref{no4}), the following chain of
inequalities holds:%
\begin{eqnarray}
\Vert y-p\Vert_{\infty} & &
 =\Vert y-z+z-p\Vert_{\infty}\leq
 \Vert y-z\Vert_{\infty}+\Vert z-p \Vert_{\infty}\nonumber\\
& &  \leq\beta(\Vert x-q\Vert_{\infty},\tau)+\gamma(\Vert l_{1}-l_{2}\Vert_{\infty}%
)+ \lambda_{\mathcal{X}} \nonumber\\
& & \leq\beta(\varepsilon,\tau)+\gamma
(\lambda_{\mathcal{U}})+ \lambda_{\mathcal{X}}\leq \varepsilon.
\label{no7}
\end{eqnarray}
Hence $(y,p)\in R$ and condition (ii) in Definition \ref{ASR} holds. Condition (iii) can be shown by using a similar reasoning.
%
Finally by the inequality in (\ref{cond2}) and the definition of $q_{2}^{0}$, $\Vert \xi_{0}- q_{2}^{0} \Vert \leq \lambda_{\mathcal{X}} \leq \varepsilon$ and hence, condition (iv) is also satisfied.
\end{proof}

The above result relies upon existence of a countable approximation for the set of reachable states. In order to address this issue, we consider one possible approximation scheme of functional spaces based on spline analysis \cite{SplineBook}. Spline based approximation schemes have been extensively used in the literature of time--delay systems (see e.g. \cite{GermaniSIAM00} and the references therein). 
Let us consider the space $\mathcal{Y}\subseteq C^{0}(I,Y)$ with $Y\subseteq\mathbb{R}^{n}$, $I=[a,b]$, $a,b\in\mathbb{R}$ and $a<b$. 
Given $N\in \mathbb{N}$ consider the following functions (see \cite{SplineBook}):
\[
\begin{array}{rllll}
s_{0}(t)= & 
\left\{
\begin{array}
[c]{lll}
1-(t-a)/h, & t\in [a,a+h],&\\ 
0, & \textrm{otherwise,} &
\end{array}
\right.
& & &  \\ \\
s_{i}(t)= 
&
\left\{
\begin{array}
[c]{lll}
1-i+(t-a)/h, & t\in [a+(i-1)h,a+ih], & \\
1+i-(t-a)/h, & t\in [a+ih,a+(i+1)h], & \\
0, & \textrm{otherwise,} & i=1,2,...,N; 
\end{array} 
\right.\\ \\
s_{N+1}(t)=
&
\left\{
\begin{array}
[c]{lll}
1+(t-b)/h, &  t\in [b-r,b], & \\
0, & \textrm{otherwise,} &
\end{array}
\right.
 & &  \\
\end{array}
\]

where $h=(b-a)/(N+1)$. Functions $s_{i}$ called \textit{splines}, are used to approximate $\mathcal{Y}$. 
The approximation scheme that we use is composed of two steps: we first approximate a function $y\in\mathcal{Y}$ (Figure \ref{fig1}; upper panel) by means of the piecewise--linear function $y_{1}$ (Figure \ref{fig1}; medium panel), obtained by the linear combination of the $N+2$ splines $s_{i}$, centered at time $t=a+ih$ with amplitude\footnote{This first step allows us to approximate the \textit{infinite} dimensional space $\mathcal{Y}$ by means of the \textit{finite} dimensional space $Y^{N+2}$.} $y(a+ih)$; we then approximate function $y_{1}$ by means of function $y_{2}$ (Figure \ref{fig1}; lower panel), obtained by the linear combination of the $N+2$ splines $s_{i}$, centered at time $t=a+ih$ with amplitude $\tilde{y}_{i}$ in the lattice\footnote{We recall that the set $[Y]_{2\theta}$ is defined as in (\ref{grid}).} $[Y]_{2\theta}$, which minimizes the distance from\footnote{This second step allows us to approximate the \textit{finite} dimensional space $Y^{N+2}$ by means of the \textit{countable} space $([Y]_{2\theta})^{N+2}$, which becomes a finite set when the set $Y$ is bounded.} $y(a+ih)$, i.e. 
\mbox{$\tilde{y}_{i}=\arg \min_{y\in [Y]_{2\theta}}\Vert y-y(a+ih) \Vert $}.

\begin{figure}
\begin{center}
\includegraphics[scale=0.32]{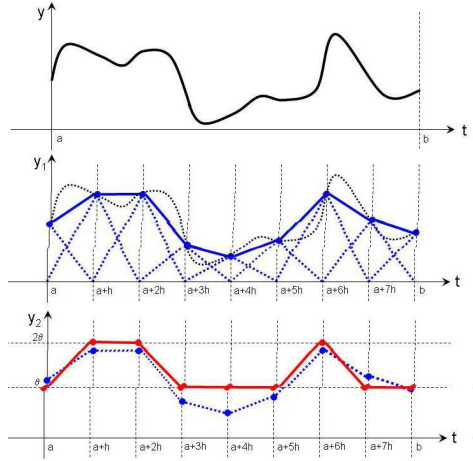}
\caption{Spline--based approximation scheme of a functional space.}
\label{fig1}
\end{center}
\end{figure}
Given any $N\in\mathbb{N}$, $\theta,M\in \mathbb{R}^{+}$ let\footnote{The real $M$ is a parameter associated with $\mathcal{Y}$ and its role will become clear in the subsequent developments.}:
\begin{equation}
\Lambda(N,\theta,M):=h^{2} M/8+(N+2)\theta,
\label{lambda}
\end{equation}
with $h=(b-a)/(N+1)$. Function $\Lambda$ will be shown to be an upper bound to the error associated with the approximation scheme that we propose. 
It is readily seen that for any $\lambda\in\mathbb{R}^{+}$ and any $M\in\mathbb{R}^{+}$ there always exist $N\in\mathbb{N}$ and $\theta\in\mathbb{R}^{+}$ so that $\Lambda(N,\theta,M)\leq \lambda$. Let $N_{\lambda,M}$ and $\theta_{\lambda,M}$ be such that $\Lambda(N_{\lambda,M},\theta_{\lambda,M},M)\leq \lambda$. For any $\lambda\in\mathbb{R}^{+}$ and $M\in\mathbb{R}^{+}$, define the operator 
\[
\psi_{\lambda,M}: \mathcal{Y} \rightarrow C^{0}([a,b];Y),
\]
that associates to any function $y\in\mathcal{Y}$ the function:
\begin{equation}
\psi_{\lambda,M}(y)(t):=\sum_{i=0}^{N_{\lambda,M}+1} \tilde{y}_{i}s_{i}(t), \hspace{5mm} t\in [a,b],
\label{refpsi}
\end{equation}
where $\tilde{y}_{i}\in [Y]_{2\theta_{\lambda,M}}$ and $\Vert \tilde{y}_{i} - y(a+ih) \Vert \leq \theta_{\lambda,M}$, for any \mbox{$i=0,1,...,N_{\lambda,M}+1$}. Note that operator $\psi_{\lambda,M}$ is not uniquely defined. For any given $M\in\mathbb{R}^{+}$ and any given precision $\lambda\in\mathbb{R}^{+}$ define:
\begin{equation}
\mathcal{A}_{\mathcal{Y},M}(\lambda):=\psi_{\lambda,M}(\mathcal{Y}).
\label{Ay}
\end{equation}

The above approximation scheme is employed to construct countable approximations of the set $R_{\tau}(\Sigma)$ of reachable states (see Proposition \ref{coroll}). \\
Consider a digital time--delay system \mbox{$\Sigma=(X,\mathcal{X},\xi_{0},U,\mathcal{U}_{\tau},f)$} and suppose that:
\begin{itemize}
\item[(A.1)] $\Sigma$ is $\delta$--ISS;
\item[(A.2)] $X$ and $U$ are bounded sets;
\item[(A.3)] Functional $f$ is Frech\'et differentiable in
$C^{0}([-\Delta,0];\mathbb{R}^{n})\times \mathbb{R}^{m}$;
\item[(A.4)] The Frech\'et differential $J(\phi,u)$ of $f$ is bounded on bounded subsets of $C^{0}([-\Delta,0];\mathbb{R}^{n})\times \mathbb{R}^{m}$.
\end{itemize}
Under the above assumptions, the following bounds are well defined:
\begin{equation}
\begin{array}
{ll}
B_{X}=\sup_{x\in X}\Vert x \Vert , &
B_{J}=\sup_{(\phi,u)\in C^{0}([-\Delta,0];X)\times U }\Vert J(\phi,u)\Vert,\nonumber\\ 
B_{U}=\sup_{u\in U}\Vert u \Vert, &
M=(\beta (B_{X},0)+\gamma (B_{U})+B_{U})\kappa B_{J},
\end{array}
\label{cost}
\end{equation}
%
%
where 
$\kappa$ is the Lipschitz constant of functional $f$ in the bounded set $C^{0}([-\Delta,0];$ $X)\times U$ and $\Vert J(\phi,u)\Vert$ denotes the norm of the operator $J(\phi,u):C^{0}([-\Delta,0];$ $\mathbb{R}^{n})\times \mathbb{R}^{m}\rightarrow \mathbb{R}^{n}$. 
We can now give the following result that points out sufficient conditions for the existence of countable approximations of $R_{\tau}(\Sigma)$.
\begin{proposition}
Consider a digital time--delay system $\Sigma=(X,\mathcal{X},\xi_{0},U,\mathcal{U}_{\tau},f)$, satisfying assumptions (A.1-4) and the following conditions:

\[
\begin{array}
{llll}
(A.5) &
\xi_0\in PC^2([-\Delta,0];X), &
\Vert \xi_{0} \Vert_{\infty}\leq B_{X}^{0}\leq B_{X}, & 
\left \Vert D^2\xi_0\right \Vert_{\infty}<M,\nonumber\\
&
\beta(B_{X}^{0},0)+\gamma(B_{U})\leq B_{X},& 
\beta(B_{X}^{0},\tau)+\gamma(B_{U})\leq B_{X}^{0}, &
\tau>2\Delta,\nonumber
\end{array}
\nonumber
\]

with $M$ as in (\ref{cost}). Then the set $\mathcal{A}_{\mathcal{X}}$ defined for any $\lambda_{\mathcal{X}}\in\mathbb{R}^{+}$ by:
\begin{equation}
\mathcal{A}_{\mathcal{X}}(\lambda_{\mathcal{X}})=\psi_{\lambda_{\mathcal{X}},M}(R_{\tau}(\Sigma)),
\label{ApproxR}
\end{equation}
with $\psi_{\lambda_{\mathcal{X}},M}$ as in (\ref{refpsi}), is a countable approximation of $R_{\tau}(\Sigma)$.
\label{coroll}
\end{proposition}

\begin{algorithm}[h]\label{alg}
\SetLine
\caption{Construction of symbolic models for time--delay systems.}
\textbf{input:}\\
time--delay system $\Sigma=(X,\mathcal{X},\xi_{0},U,\mathcal{U},f)$ satisfying assumptions (A.1-5)\;
parameters $\tau,N,\theta,\lambda_{\mathcal{U}},M$\;
\textbf{init:}\\
$k:=0$\;
$Q^{k}:=\{q_{2}^{0}\}$, where $q_{2}^{0}=\psi_{\lambda,M}(\xi_{0})$, with $\psi_{\lambda,M}$ defined as in (\ref{refpsi}) and $\lambda=\Lambda(N,\theta,M)$\;
$Q^{k-1}:=\varnothing$\;
$\rTo_{k}:=\varnothing$\;
$H_{2}:=\imath : Q_{2} \hookrightarrow O_{2}$\;
$h:=\Delta/(N+1)$\;
\While{$Q^{k}\neq Q^{k-1}$}
{
\ForEach{$q \in Q^{k}$}
			{
			\ForEach{$l_{2}\in [U]_{2\lambda_{\mathcal{U}}}$}
				{
				\textbf{compute} $z:=x_{\tau}(q,l_{2})$\;
				\textbf{compute} $p=\psi_{\lambda,M}(z)$, with $\psi_{\lambda,M}$ defined as in (\ref{refpsi}) and $\lambda=\Lambda(N,\theta,M)$;
				$Q^{k+1}:=Q^{k}\cup\{p\}$\;
				$\rTo_{k+1}:=\rTo_{k}\cup\{(q,l_{2},p)\}$\;
				}			
			}
			k:=k+1\;
			}
\textbf{output:} $T_{\tau,N,\theta,\lambda_{\mathcal{U}}}(\Sigma):=(Q^{k},q_{2}^{0},[U]_{\lambda_{\mathcal{U}}},\rTo_{k},\mathcal{X},H_{2})$
\end{algorithm}

While the first and third assumptions in (A.5) are given apriori on the time--delay system under study, the other assumptions are satisfied for sufficiently large values of $B_{X}$ and $\tau$. 
Note that if $B_{X}$ does not satisfy the assumptions in (A.5) one can always embed the state space $X$ of the time--delay system in a bigger state space $X'$, so that the corresponding bound $B_{X'}=\sup_{x\in X'}\Vert x \Vert$ satisfies the required assumptions. The proof of the above result requires some technicalities and is therefore reported in the Appendix (Section \ref{App2}). \\
We now have all the ingredients to define a symbolic model for digital time--delay systems. Given $\tau\in\mathbb{R}^{+}$, $\theta,\lambda_{\mathcal{U}}\in\mathbb{R}^{+}$ and $N\in\mathbb{N}$, consider the transition system 
\begin{equation}
T_{\tau,N,\theta,\lambda_{\mathcal{U}}}(\Sigma):=(Q_{2},q_{2}^{0},L_{2},\rTo_{2},O_{2},H_{2}),
\label{sm}
\end{equation}
where:
\begin{itemize}
\item $Q_{2}=\mathcal{A}_{\mathcal{X}}(\Lambda(N,\theta,M))$ with $\mathcal{A}_{\mathcal{X}}$ as in (\ref{ApproxR}) with $\lambda_{\mathcal{X}}=\Lambda(N,\theta,M)$ and $M$ as in (\ref{cost});
\item $q_{2}^{0}=\psi_{\lambda,M}(\xi_{0})$, with $\psi_{\lambda_{\mathcal{X}},M}$ defined as in (\ref{refpsi}) and $\lambda_{\mathcal{X}}=\Lambda(N,\theta,M)$;
\item $L_{2}=\mathcal{A}_\mathcal{U}(\lambda_{\mathcal{U}})$; 
\item $q \rTo^{l}_{2} p$, if $\left\Vert p-x_{\tau}(q,l)\right\Vert_{\infty} 
\leq \Lambda(N,\theta,M)$;
\item $O_{2}=\mathcal{X}$;
\item $H_{2}=\imath : Q_{2} \hookrightarrow O_{2}$.
\end{itemize}
Note that the transition system in (\ref{sm}) coincides with the one in (\ref{T2}) by setting $\lambda_{\mathcal{X}}=\Lambda(N,\theta,M)$. Moreover, it is readily seen that:
\begin{proposition}
If the digital time--delay system $\Sigma$ satisfies assumptions (A.1-5), transition system $T_{\tau,N,\theta,\lambda_{\mathcal{U}}}(\Sigma)$ in (\ref{sm}) is symbolic.
\end{proposition}

Transition system $T_{\tau,N,\theta,\lambda_{\mathcal{U}}}(\Sigma)$ can be constructed by analytical and/or numerical integration of the solutions of the time--delay system. One possible construction scheme is illustrated in Algorithm \ref{alg} which proceeds, as follows. The set $Q^{k}$ of states of the symbolic model at step $k=0$ is initialized to contain the (only) symbol $q_{2}^{0}=\psi_{\lambda,M}(\xi_{0})$ that is associated with the initial condition $\xi_{0}$. Then, for any initial condition $q\in Q^{k}$ and any control input \mbox{$l_{2}\in [U]_{2\lambda_{\mathcal{U}}}$}, the algorithm computes the solution $z=x_{\tau}(q,l_{2})$ of the differential equation in (\ref{TDS}) at time $t=\tau$, and it adds the symbol $p=\psi_{\lambda,M}(z)$ to $Q^{k}$. In the end of this basic step, index $k$ is increased to $k+1$ and the above basic step is repeated. 
The algorithm continues by adding symbols to $Q^{k}$ since no more symbols are found, or equivalently, since a step $k^{*}$ is found, for which $Q^{k^{*}}=Q^{k^{*}+1}$. Convergence properties of Algorithm \ref{alg} are discussed in the following result.

\begin{theorem}
Algorithm \ref{alg} terminates in a finite number of steps.
\end{theorem}
\begin{proof}
Let $\mathcal{Z}$ be the collection of all functions of the form (\ref{refpsi}) with \mbox{$\tilde{y}_{1},\tilde{y}_{2},$ $...,\tilde{y}_{N_{\lambda,M}+1}\in [X]_{2\theta}$}. 
%
Since the set $X$ is bounded, the set $[X]_{2\theta}$ is finite and hence the set $\mathcal{Z}$ is finite as well. 
By construction the sequence $Q^{k}$ is non--decreasing, i.e. $Q^{k}\subseteq Q^{k+1}$ and each set of the sequence is contained in $\mathcal{Z}$, i.e. $Q^{k}\subseteq \mathcal{Z}$. Hence, a fixed point of Algorithm \ref{alg} will be found in a finite number of steps, which is upper bounded by the cardinality of $\mathcal{Z}$.
\end{proof}

We can now give the main result of this paper.

\begin{theorem}
Consider a digital time--delay system $\Sigma=(X,\mathcal{X},\xi_{0},U,\mathcal{U}_{\tau},f)$ and any desired precision $\varepsilon\in\mathbb{R}^{+}$. Suppose that assumptions (A.1-5) are satisfied. Moreover let 
$\tau,\theta,\lambda_{\mathcal{U}}\in\mathbb{R}^{+}$ and $N\in\mathbb{N}$ satisfy the following
inequality%
\begin{eqnarray}
\beta(\varepsilon,\tau)+\gamma(\lambda_{\mathcal{U}})+ \Lambda(N,\theta,M) \leq \varepsilon,
\label{cond3}
\end{eqnarray}
with $\Lambda$ as in (\ref{lambda}) and $M$ as in (\ref{cost}). Then transition systems $T_{\tau}(\Sigma)$ and $T_{\tau,N,\theta,\lambda_{\mathcal{U}}}(\Sigma)$ are $\varepsilon$--bisimilar.
\label{Main}
\end{theorem}

\begin{proof}
The map $\mathcal{A}_{\mathcal{U}}$ is a countable approximation of $U$ and by Proposition \ref{coroll}, the map $\mathcal{A}_{\mathcal{X}}$ is a countable approximation of $R_{\tau}(\Sigma)$. Choose $\lambda_{\mathcal{X}} \in \mathbb{R}^{+}$ and $\lambda_{\mathcal{U}}\in\mathbb{R}^{+}$ satisfying 
the inequality in (\ref{cond2}). There exist $\theta\in\mathbb{R}^{+}$ and $N\in\mathbb{N}$ so that $\lambda_{\mathcal{X}}=\Lambda(N,\theta,M)$ and hence the inequality in (\ref{cond3}) holds. Finally the result holds as a direct application of Theorem \ref{ThMain}.
\end{proof}

The above result is important because it provides a method to translate time--delay systems to approximately bisimilar symbolic models. Hence, it gives a concrete alternative methodology to the control design of nonlinear time--delay systems, by regarding control design on time--delay systems as control design on symbolic models (see e.g. \cite{LTLControl,TabuadaTAC08} for symbolic models--based control of linear and nonlinear control systems).

\section{Conclusion}
In this paper we showed that incrementally input--to--state stable digital time--delay systems admit symbolic models that are approximately bisimilar to the original system, with a precision that can be rendered as small as desired. An algorithm has been presented which computes the proposed symbolic models. Termination of the algorithm in finite time is ensured under a boundedness assumption on the sets of states and inputs of the system.
 
\bibliographystyle{alpha}
\bibliography{biblio1}

\begin{thebibliography}{BBW08}

\bibitem[Ang02]{IncrementalS}
D.~Angeli.
\newblock A {L}yapunov approach to incremental stability properties.
\newblock {\em IEEE Transactions on Automatic Control}, 47(3):410--421, 2002.

\bibitem[BBW08]{Batt:08}
G.~Batt, C.~Belta, and R.~Weiss.
\newblock Temporal logic analysis of gene networks under parameter uncertainty.
\newblock {\em IEEE Transactions of Automatic Control}, 53:215--229, 2008.

\bibitem[BH06]{Belta:06}
C.~Belta and L.C.G.J.M. Habets.
\newblock Controlling a class of nonlinear systems on rectangles.
\newblock {\em IEEE Transactions of Automatic Control}, 51(11):1749--1759,
  2006.

\bibitem[BIP05]{Belta:05}
C.~Belta, V.~Isler, and G.~J. Pappas.
\newblock Discrete abstractions for robot planning and control in polygonal
  environments.
\newblock {\em IEEE Transactions on Robotics}, 21(5):864--874, 2005.

\bibitem[Dri62]{Driver62}
R.~D. Driver.
\newblock Existence and stability of solutions of a delay-differential system,.
\newblock {\em Archive for Rational Mechanics and Analysis}, 10:401--426, 1962.

\bibitem[EFP06]{TACsymbolicmodels}
M.~Egerstedt, E.~Frazzoli, and G.~J. Pappas.
\newblock {\em IEEE Transactions of Automatic Control}, 51(6), June 2006.
\newblock Special Issue on Symbolic Methods for Complex Control Systems.

\bibitem[GMP00]{GermaniSIAM00}
A.~Germani, C.~Manes, and P.~Pepe.
\newblock A twofold spline approximation for finite horizon {LQG} control of
  hereditary systems.
\newblock {\em SIAM Journal on Control and Optimization}, 39(4):1233--1295,
  2000.

\bibitem[GP07]{AB-TAC07}
A.~Girard and G.J. Pappas.
\newblock Approximation metrics for discrete and continuous systems.
\newblock {\em IEEE Transactions on Automatic Control}, 52(5):782--798, 2007.

\bibitem[GPT08]{Girard_et_al_HSCC08}
A.~Girard, G.~Pola, and P.~Tabuada.
\newblock Approximately bisimilar symbolic models for incrementally stable
  switched systems.
\newblock In M.~Egerstedt and B.~Mishra, editors, {\em Hybrid Systems:
  Computation and Control}, volume 4981 of {\em Lecture Notes in Computer
  Science}, pages 201--214. Springer Verlag, Berlin, 2008.

\bibitem[KPJ08]{PepeEJC08}
I.~Karafyllis, P.~Pepe, and Z.~P. Jiang.
\newblock Input-to-output stability for systems described by retarded
  functional differential equations.
\newblock {\em European Journal of Control}, 14(6):539--555, December 2008.

\bibitem[Mil89]{Milner}
R.~Milner.
\newblock {\em Communication and Concurrency}.
\newblock Prentice Hall, 1989.

\bibitem[Nic01]{Niculescu:01}
S.~I. Niculescu.
\newblock {\em Delay Effects on Stability, a Robust Control
  ApproachIntroduction to the Theory and Applications of Functional
  Differential Equations}.
\newblock Lecture Notes in Control and Information Sciences. Springer, London,
  2001.

\bibitem[Par81]{Park}
D.M.R. Park.
\newblock Concurrency and automata on infinite sequences.
\newblock volume 104 of {\em Lecture Notes in Computer Science}, pages
  167--183, 1981.

\bibitem[Pep07a]{PepeSAFEAUTOMATICA}
P.~Pepe.
\newblock On {L}iapunov-{K}rasovskii {F}unctionals under {C}arath\'eodory
  {C}onditions.
\newblock {\em Automatica}, 43(4):701--706, 2007.

\bibitem[Pep07b]{PepeTACABS}
P.~Pepe.
\newblock The {P}roblem of the {A}bsolute {C}ontinuity for
  {L}iapunov-{K}rasovskii {F}unctionals.
\newblock {\em IEEE Transactions on Automatic Control}, 52(5):953--957, 2007.

\bibitem[PGT08]{PolaAutom2008}
G.~Pola, A.~Girard, and P.~Tabuada.
\newblock Approximately bisimilar symbolic models for nonlinear control
  systems.
\newblock {\em Automatica}, 44:2508--2516, October 2008.

\bibitem[PJ06]{PepeJiangSCL}
P.~Pepe and Z.~P. Jiang.
\newblock A {L}yapunov-{K}rasovskii {M}ethodology for {ISS} and {iISS} of
  time-delay systems.
\newblock {\em Systems \& Control Letters}, 55(12):1006--1014, 2006.

\bibitem[Ric03]{surveyTDS}
J.~P. Richard.
\newblock Time-delay systems: an overview of some recent advances and open
  problems.
\newblock {\em Automatica}, 39(10):1667--1694, October 2003.

\bibitem[Sch73]{SplineBook}
M.~H. Schultz.
\newblock {\em Spline Analysis}.
\newblock Prentice Hall, 1973.

\bibitem[Tab08]{TabuadaTAC08}
P.~Tabuada.
\newblock An approximate simulation approach to symbolic control.
\newblock {\em IEEE Transactions on Automatic Control}, 53(6):1406--1418, 2008.

\bibitem[TDS07]{TDSAQ}
In C.~Manes and P.~Pepe, editors, {\em Proceedings of the 6th IFAC Workshop on
  Time-Delay Systems}, volume~6. IFAC-PapersOnline, 2007.

\bibitem[TP06]{LTLControl}
P.~Tabuada and G.J. Pappas.
\newblock {L}inear {T}ime {L}ogic control of discrete-time linear systems.
\newblock {\em IEEE Transactions on Automatic Control}, 51(12):1862--1877,
  2006.

\end{thebibliography}

\section{Appendix}
\subsection{Notation}\label{sec:notation}
The symbols $\mathbb{N}$, $\mathbb{Z}$, $\mathbb{R}$, $\mathbb{R}^{+}$ and $\mathbb{R}_{0}^{+}$ denote the sets of natural, integer, real, positive and nonnegative real numbers, respectively. 
Given a vector $x\in\mathbb{R}^{n}$ the $i$--th element of $x$ is denoted by $x_{i}$; furthermore $\Vert x\Vert$ denotes the infinity norm of $x$; we recall that 
\mbox{$\Vert x\Vert:=max\{|x_1|,|x_2|,...,|x_n|\}$}, where $|x_i|$ is the absolute value of $x_i$. 
For any $A\subseteq
\mathbb{R}^{n}$ and \mbox{$\theta\in{\mathbb{R}^{+}}$} define 
\begin{equation}
[A]_{\theta}:=\{a\in A\,\,|a_{i}=k_{i}\theta,\,\,\,k_{i}\in\mathbb{Z},i=1,...,n\}.
\label{grid}
\end{equation}
Given a measurable and locally essentally bounded function \mbox{$f:\mathbb{R}_{0}^{+}\rightarrow\mathbb{R}^{n}$}, the
\mbox{(essential)} supremum norm of $f$ is denoted by $\Vert f\Vert_{\infty}$; we recall that 
\mbox{$\Vert f\Vert_{\infty}:=(ess)sup\{\Vert f(t)\Vert,$ $t\geq0\}$}. 
For a given time $\tau\in\mathbb{R}^{+}$, define $f_{\tau}$ so that
$f_{\tau}(t)=f(t)$, for any $t\in [0,\tau[$, and $f(t)=0$ elsewhere; 
$f$ is said to be locally essentially bounded if for any $\tau\in\mathbb{R}^{+}$,
$f_{\tau}$ is essentially bounded.
A continuous function $\gamma:\mathbb{R}_{0}^{+}%
\rightarrow\mathbb{R}_{0}^{+}$ is said to belong to class $\mathcal{K}$ if it
is strictly increasing and \mbox{$\gamma(0)=0$}; $\gamma$ is said to belong to class
$\mathcal{K}_{\infty}$ if \mbox{$\gamma\in\mathcal{K}$} and $\gamma(r)\rightarrow
\infty$ as $r\rightarrow\infty$. A continuous function \mbox{$\beta:\mathbb{R}_{0}^{+}\times\mathbb{R}_{0}^{+}\rightarrow\mathbb{R}_{0}^{+}$} is said to
belong to class $\mathcal{KL}$ if for each fixed $s$, the map $\beta(r,s)$
belongs to class $\mathcal{K}$ with respect to $r$ and, for each
fixed $r$, the map $\beta(r,s)$ is decreasing with respect to $s$ and
$\beta(r,s)\rightarrow0$ as \mbox{$s\rightarrow\infty$}. 
Given $k,n\in\mathbb{N}$ with $n\geq 1$ and $I=[a,b]\subseteq \mathbb{R}$, $a,b\in\mathbb{R}$, $a<b$ let $C^{k}(I;\mathbb{R}^{n})$ be the space of functions $f:I\rightarrow \mathbb{R}^{n}$ that are continuously differentiable $k$ times. Given $k\geq 1$, let $PC^{k}(I;\mathbb{R}^{n})$ be the space of $C^{k-1}(I;\mathbb{R}^{n})$ functions $f:I\rightarrow \mathbb{R}^{n}$ whose $k$--th derivative exists except in a finite number of reals, and it is bounded, i.e. there exist $\gamma_{0},\gamma_{1},...,\gamma_{s}\in\mathbb{R}^{+}$ with  $a=\gamma_{0}<\gamma_{1}<...<\gamma_{s}=b$ so that $D^{k}\,f$ is defined on each open interval $(\gamma_{i},\gamma_{i+1})$, $i=0,1,...,s-1$ and $\max_{i=0,1,...,s-1}\sup_{t\in(\gamma_{i},\gamma_{i+1})}\Vert D^{k}\,f(t) \Vert _{\infty}<\infty$.  
For any continuous function $x(s)$, defined on
$-\Delta\leq s<a$, $a>0$, and any fixed $t$, $0\leq t<a$, the
standard symbol $x_{t}$ will denote the
 element of $C^{0}([-\Delta,0];\mathbb{R}^{n})$ defined by
$x_{t}(\theta)=x(t+\theta)$, $-\Delta\leq\theta\leq 0$. 
The identity map on a set $A$ is denoted by $1_{A}$. 
Given two sets $A$ and $B$, if $A$ is a subset of $B$
we denote by \mbox{$\imath_{A}:A\hookrightarrow B$} or simply by $\imath$ the natural
inclusion map taking any $a\in A$ to \mbox{$\imath (a) = a \in B$}. Given a function $f:A\rightarrow B$ the symbol $f(A)$ denotes
the image of $A$ through $f$, i.e. $f(A):=\{b\in B:\exists a\in A$ s.t.
$b=f(a)\}$. 
\subsection{Technical Proofs}\label{App2}
The proof of Proposition \ref{coroll} is based on the following lemmas.

\begin{lemma}
Suppose that $\mathcal{Y}\subseteq PC^{2}(I;Y)$ and there exists $M\in\mathbb{R}^{+}$ so that $\Vert D^{2}\,y\Vert_{\infty}\leq M$ for any $y\in \mathcal{Y}$. Then $\mathcal{A}_{\mathcal{Y},M}$ as defined in (\ref{Ay}), is a countable approximation of $\mathcal{Y}$.
\label{prop1}
\end{lemma}

\begin{proof}
Let $\mathcal{Z}$ be the collection of all functions of the form (\ref{refpsi}) with \mbox{$\tilde{y}_{1},\tilde{y}_{2},$ $...,\tilde{y}_{N_{\lambda,M}+1}\in [Y]_{2\theta_{\lambda,M}}$}. By construction, since $\psi_{\lambda,M}(\mathcal{Y})$ is a subset of $\mathcal{Z}$ that is countable, it is countable as well. Hence, condition (i) in Definition \ref{CountApprox} is satisfied. Let us now show that also condition (ii) is satisfied. Consider any $\lambda\in\mathbb{R}^{+}$ and any $y\in \mathcal{Y}$ and set $h_{\lambda,M}=(b-a)/(N_{\lambda,M}+1)$. By Theorem 2.6 in \cite{SplineBook} and the definition of $M$, the following inequality holds:
\begin{equation}
\Vert y- \pi \Vert_\infty \leq h_{\lambda,M}^{2} \Vert D^{2}y \Vert_{\infty}/8 \leq h_{\lambda,M}^{2} M/8,
\label{no1}
\end{equation}
where 
$
\pi(t)=\sum_{i=0}^{N_{\lambda,M}+1} y
\left(
h_{\lambda,M}i+a
\right)
s_{i}(t)$, $t\in [a,b]
$. 
Moreover by setting  
$
z(t)=\psi_{\lambda,M}(y(t))=\sum_{i=0}^{N_{\lambda,M}+1} \tilde{y}_{i}
s_{i}(t)$, $t\in [a,b]
$, 
the following chain of inequalities holds:
\begin{equation}
\begin{array}
[c]{lll}
\Vert
\pi - z\Vert_{\infty} =
&  
\left\Vert \sum_{i=0}^{N_{\lambda,M}+1}
\left(
y
\left(
h_{\lambda,M}i+a
\right)
-\tilde{y}_{i}
\right)s_{i}
\right\Vert_{\infty} \leq \\
& 
\sum_{i=0}^{N_{\lambda,M}+1} \left\Vert \left(
y
\left(
h_{\lambda,M}i+a
\right)
-\tilde{y}_{i}
\right)
s_{i} \right\Vert_{\infty} \leq \\
&
\sum_{i=0}^{N_{\lambda,M}+1} \left\Vert
y
\left(
h_{\lambda,M}i+a
\right)
-\tilde{y}_{i}
\right\Vert \Vert s_{i} \Vert_{\infty} \leq \\
&
\left(
\max_{i=0,1,...,N_{\lambda,M}+1}
 \left\Vert
y
\left(
h_{\lambda,M}i+a
\right)
-\tilde{y}_{i}
\right\Vert
\right)
\sum_{i=0}^{N_{\lambda,M}+1} \Vert s_{i} \Vert_{\infty} \leq \\
& 
\theta_{\lambda,M}(N_{\lambda,M}+2).\label{no2} 
\end{array}
\end{equation}
By combining inequalities in (\ref{no1}) and in (\ref{no2}) and by definition of 
$\theta_{\lambda,M}$ and $N_{\lambda,M}$, one gets:
\begin{eqnarray}
\Vert y-z  \Vert_{\infty} & \leq & \Vert y-\pi  \Vert_{\infty} +\Vert \pi-z  \Vert_{\infty} \nonumber\\
& \leq & 
h_{\lambda,M} ^{2} M/8+
\theta_{\lambda,M}(N_{\lambda,M}+2) 
= 
\Lambda(N_{\lambda,M},\theta_{\lambda,M},M) \leq \lambda.\nonumber
\end{eqnarray}
Hence, condition (ii) in Definition \ref{CountApprox} is satisfied. We conclude by showing that also condition (iii) holds. Consider any $z\in \mathcal{A}_{\mathcal{Y},M}(\lambda)$. By construction there exists $y\in\mathcal{Y}$ so that $z=\psi_{\lambda,M}(y)$. Hence, by following the same reasoning in proving condition (ii), condition (iii) can be proved as well.
\end{proof}


Under assumptions in (A.1-4), the regularity properties of the initial state in (A.5) propagate to the whole set of reachable states, or in other words, time--delay systems are invariant with respect to those properties in (A.5). More precisely:
\begin{lemma}
Consider a digital time--delay system $\Sigma=(X,\mathcal{X},\xi_{0},U,\mathcal{U},f)$, satisfying assumptions (A.1-5). Then for any $x_{\tau}\in R_{\tau}(\Sigma)$,  
\label{prop22}
\end{lemma}
\[
x_{\tau}\in PC^2([-\Delta,0];X), 
\qquad
\Vert x_{\tau} \Vert_{\infty}\leq B_{X}^{0}, 
\qquad
\left \Vert D^{2} x_{\tau}\right \Vert_{\infty} \le M.
\]
\begin{proof}
First note that the function $t\to \dot{x}(t)$, $t\in
[0,\tau]$, is uniformly continuous in the (compact) set $[0,\tau]$. Since $\tau>2\Delta$, it follows that
\mbox{$x_{\tau+\theta}\in C^1([-\Delta,0];X)$}, $\theta\in (-\Delta,0)$
(i.e. the derivative $\dot{x}_{\tau+\theta}$ belongs to
$C^{0}([-\Delta,0];X))$. Moreover, by taking into account the Lipschitz property of $f$, the
$\delta$--ISS inequality, the bounds on initial state and input, the following inequality holds:
\[
\begin{array}
[c]{lll}
\Vert\dot{x}_{\tau+\theta} \Vert_{\infty} & = & \sup_{\alpha\in [-\Delta,0]} 
\Vert f(x_{\tau+\theta+\alpha},u(\tau+\theta+\alpha-r)) \Vert \nonumber\\
& \leq & \kappa\sup_{\alpha\in [-\Delta,0]}(\Vert x_{\tau+\theta+\alpha}  \Vert_{\infty} + \Vert u(\tau+\theta+\alpha-r)  \Vert) \nonumber\\
& \leq & \kappa(\beta (B_{X}^{0},0)+\gamma (B_{U})+B_{U}), 
\theta\in ]-\Delta,0[.\nonumber
\end{array}
\]
As far as the second derivative is concerned, the following
equality holds, for $\theta\in ]-\Delta,0[$,
\[
\frac{d^{2}x_{\tau}(\theta)}{d\theta^{2}}=J(x_{\tau+\theta},u(\tau+\theta-r))\left
(\begin{array}{cc}\dot{x}_{\tau+\theta}
\\ 0
\end{array} \right ).\nonumber
\]
By taking into accounts the bound on the Frech\'et differential,
and the bound on the derivative $\dot{x}_{\tau+\theta}$ and $ \dot
u(t)=0$, we obtain $\Vert D^{2} x_{\tau} \Vert_{\infty}\le M$. Finally by assumptions of $B_{X}^{0}$, $B_{U}$ and $\tau$ in (A.5), it is readily seen that $\Vert x_{\tau} \Vert_{\infty}\leq B_{X}^{0}$.
\end{proof}

By combining Lemmas \ref{prop1} and \ref{prop22}, the proof of Proposition \ref{coroll} holds as a direct consequence.

\end{document}